\documentclass[12pt]{amsart}

\usepackage{enumerate}
\usepackage{etex}
\reserveinserts{28}
\usepackage[latin1]{inputenc}
\usepackage{amsmath}
\usepackage{amsfonts}
\usepackage{amssymb}
\usepackage{mathtools}
\usepackage{color}                    
\usepackage{hyperref}                 
\usepackage[T1]{fontenc}

\newtheorem{thm}{Theorem}[section]
\newtheorem{cor}[thm]{Corollary}
\newtheorem{lem}[thm]{Lemma}
\newtheorem{prop}[thm]{Proposition}

\theoremstyle{definition}
\newtheorem{defn}[thm]{Definition}

\newtheorem{prob}[thm]{Problem}

\numberwithin{equation}{section}

\newtheorem{theoremA}{Theorem} 
\newtheorem{theoremB}{Theorem}
\newtheorem{theoremC}{Theorem}

\def\({\left(}
\def\){\right)}

\def\N{\mathbb{ N}}

\def\1{\textbf{1}}

\def\supp{\operatorname{supp}}

\def\spn{\operatorname{span}}

\newcommand{\norm}[1]{\left\lVert#1\right\rVert}

\title[Linear versus lattice embeddings]{Linear versus lattice embeddings between Banach lattices}

\author[Avil\'es]{Antonio Avil\'es}
\address[Avil\'es]{Universidad de Murcia, Departamento de Matem\'{a}ticas, Campus de Espinardo 30100 Murcia, Spain
	\newline
	\href{https://orcid.org/0000-0003-0291-3113}{ORCID: \texttt{0000-0003-0291-3113} } }
\email{\texttt{avileslo@um.es}}

\author[Mart\'inez-Cervantes]{Gonzalo Mart\'inez-Cervantes}
\address[Mart\'inez-Cervantes]{Universidad de Murcia, Departamento de Matem\'{a}ticas, Campus de Espinardo 30100 Murcia, Spain
	\newline
	\href{http://orcid.org/0000-0002-5927-5215}{ORCID: \texttt{0000-0002-5927-5215} } }	
\email{gonzalo.martinez2@um.es}

\author[Rueda Zoca]{Abraham Rueda Zoca}
\address[Rueda Zoca]{Universidad de Murcia, Departamento de Matem\'{a}ticas, Campus de Espinardo 30100 Murcia, Spain
	\newline
	\href{https://orcid.org/0000-0003-0718-1353}{ORCID: \texttt{0000-0003-0718-1353} }}
\email{\texttt{abraham.rueda@um.es}}
\urladdr{\url{https://arzenglish.wordpress.com}}

\author[P. Tradacete]{Pedro Tradacete}
\address[Tradacete]{Instituto de Ciencias Matem\'aticas (CSIC-UAM-UC3M-UCM)\\
	Consejo Superior de Investigaciones Cient\'ificas\\
	C/ Nicol\'as Cabrera, 13--15, Campus de Cantoblanco UAM\\
	28049 Madrid, Spain.
	\newline
	\href{http://orcid.org/0000-0001-7759-3068}{ORCID: \texttt{0000-0001-7759-3068} }}
\email{pedro.tradacete@icmat.es}



\begin{document}

\begin{abstract}
A well-known classical result states that $c_0$ is linearly embeddable in a Banach lattice if and only if it is lattice embeddable. Improving results of H.P.~Lotz, H.P.~Rosenthal and N.~Ghoussoub, we prove that $C[0,1]$ shares this property with $c_0$. Furthermore, we show that any infinite-dimensional sublattice of $C[0,1]$ is either lattice isomorphic to $c_0$ or contains a sublattice isomorphic to $C[0,1]$. As a consequence, it is proved that for a separable Banach lattice $X$ the following conditions are equivalent:
\begin{enumerate}
	\item $X$ is linearly embeddable in a Banach lattice if and only if it is lattice embeddable;
	\item $X$ is lattice embeddable into $C[0,1]$.
\end{enumerate} 
\end{abstract}

\maketitle

\section{Introduction}

%
%
%


A Banach lattice is a Banach space equipped with compatible lattice operations. Understanding the relation between the linear and lattice structures of Banach lattices has been the driving force behind a large part of research in the topic. However, many fundamental questions concerning this relation still remain open. 

It is well-known that given a Banach lattice $X$, the class of Banach lattices which can arise as a \emph{sublattice} of $X$ can be very different from those which can arise as a \emph{subspace} of $X$. Nevertheless, a number of results relating the containment of a Banach lattice as a subspace to the containment as a sublattice can be found in the literature (see \cite[\textsection 4.3]{AB}). For instance, a well-known fact, attributed to P.~Meyer-Nieberg \cite{MN1,MN2}, states that the Banach lattice $c_0$ is linearly embeddable in a Banach lattice if and only if it is lattice embeddable (cf. \cite[Theorem 4.61]{AB}). 

Under additional assumptions, similar results on this line are known: $\ell_1$ embeds as a sublattice of a Banach lattice $X$ if and only if it embeds as a complemented subspace (cf. \cite[Theorem 4.69]{AB}); if $X$ is a Dedekind $\sigma$-complete Banach lattice, then $\ell_\infty$ embeds as a sublattice of $X$ precisely when it embeds as a subspace (cf. \cite[Theorem 4.56]{AB}). In connection with these facts, a classical result attributed to G. Ja. Lozanovski \cite{Loza} states that neither $c_0$ nor $\ell_1$ are lattice embeddable in a Banach lattice $X$ if, and only if, neither of them linearly embeds in $X$ (which actually provides a characterization of reflexivity, cf. \cite[Theorem 4.71]{AB}).
Finally, let us also mention the deep result due to N. Kalton that if $X$ is an order-continuous rearrangement-invariant Banach lattice on $[0,1]$ different from $L_2$, then $X$ embeds as a complemented subspace of an order-continuous Banach lattice $Y$ if and only if $X$ embeds as a complemented sublattice \cite{Kalton}.

%

The problem we address in the present paper is to characterize those Banach lattices which embed as a sublattice whenever they embed as a subspace in an arbitrary Banach lattice. To our knowledge, no example other than $c_0$ can be found in the literature. The main result of the article is the following:

\begin{theoremA}
\label{TheoA}
The Banach lattice $C[0,1]$ is linearly embeddable in a Banach lattice if and only if it is lattice embeddable.
\end{theoremA}

The proof of Theorem \ref{TheoA} relies on some classical and recent results in the theory of Banach lattices. On the one hand, H. P. Lotz and H. P. Rosenthal proved in \cite{LotzRosenthal78} that for a Banach lattice $X$ it is equivalent that: 
\begin{enumerate}
	\item $L_1$ is lattice embeddable in $X^*$; 
	\item $C(\Delta)^*$ is lattice embeddable in $X^*$;
	\item there exists a positive embedding $T\colon C(\Delta) \longrightarrow X$;
\end{enumerate}
where $\Delta$ denotes the Cantor space.
Furthermore, they asked whether these conditions are in turn equivalent to the fact that $C(\Delta)$ linearly embeds into $X$. Notice that one cannot expect that the linear embeddability of $C(\Delta)$ into $X$ implies the lattice embeddability of $C(\Delta)$. Indeed, $C(\Delta)$ linearly embeds in $C[0,1]$ (they are in fact linearly isomorphic by Miljutin's Theorem \cite[Theorem 21.5.10]{semadeni}) but $C(\Delta)$ is not lattice isomorphic to a sublattice of $C[0,1]$ (see, for example, Lemma \ref{Peanoproposition}). The previous question of Lotz and Rosenthal was answered in the affirmative by N. Ghoussoub in \cite{Ghoussoub83}. Although the proof of Theorem \ref{TheoA} relies on the aforementioned results of Lotz, Rosenthal and Ghoussoub, note that our theorem provides a strengthening of Ghoussoub's answer, since the linear embeddability of $C(\Delta)$ is equivalent to the linear embeddability of $C[0,1]$, whereas the lattice embeddability of $C[0,1]$ into $X$ implies condition (3), since there are positive linear embeddings from $C(\Delta)$ into $C[0,1]$ (for instance, because of the above mentioned result of N. Ghoussoub).

On the other hand, D.H.~Leung, L.~Li, T.~Oikhberg and M.A.~Tursi recently proved that every separable Banach lattice embeds lattice isometrically into $C(\Delta,L_1)$ \cite[Theorem 1.1]{LLOT19}, where $C(\Delta,L_1)$ denotes the Banach lattice of continuous functions $f\colon \Delta\rightarrow L_1$, endowed with the norm $\|f\|=\sup_{t \in \Delta} \|f(t)\|_{L_1}$. This result will also play a fundamental role in the proof of Theorem \ref{TheoA}.

The last key ingredient of Theorem \ref{TheoA} is the recent concept of projectivity for Banach lattices introduced by B. de Pagter and A.W. Wickstead in \cite{dPW15} (see Definition \ref{DefinitionProjectivity} below). The study of projectivity for $C(K)$-spaces was initiated by these authors and a complete characterization was given in \cite{amr20_2}. We will make use of the fact that $C[0,1]$ is projective \cite[Corollary 11.5]{dPW15} and, more precisely, of Lemma \ref{LemmEmbeddabilityC[0,1]}, which is a direct consequence of this fact.

\bigskip

Whereas Section \ref{SectionC[0,1]} is devoted to the proof of Theorem \ref{TheoA}, in Section \ref{SectionEmbeddability}, we study what other Banach lattices share the aforementioned property with $c_0$ and $C[0,1]$, providing a characterization in the separable case:

\begin{theoremB}
\label{TheoB}
For a separable Banach lattice $X$ the following conditions are equivalent:
\begin{enumerate}
	\item $X$ is linearly embeddable in a Banach lattice if and only if it is lattice embeddable;
	\item $X$ is lattice embeddable into $C[0,1]$.
\end{enumerate} 
\end{theoremB}

Theorem \ref{TheoB} will be obtained as a simple consequence of Theorem \ref{TheoA} and the following theorem, which is of independent interest and, surprisingly, seems to be new in the literature:

\begin{theoremC}
\label{TheoC}
An infinite-dimensional sublattice of $C[0,1]$ is either lattice isomorphic to $c_0$ or it contains a sublattice isomorphic to $C[0,1]$.
\end{theoremC}

We finish Section \ref{SectionEmbeddability} with some consequences of Theorem \ref{TheoB}, a simple characterization of those $C(K)$-spaces which lattice embed into $C[0,1]$, and a discussion on the possible existence of nonseparable Banach lattices for which linear embeddability implies lattice embeddability, leaving several open questions.

\bigskip

\textbf{Notation. } Our terminology is standard; any unexplained notation can be found in \cite{AB}, \cite{LinTza12} and \cite{Meyer}. By an operator we shall mean a bounded linear operator. An operator $T\colon X \longrightarrow Y$ between Banach lattices is said to be a lattice homomorphism if $Tx \vee Ty = T (x\vee y)$ and $Tx \wedge Ty = T (x\wedge y)$. We say that $T$ is a lattice isomorphism if it is a lattice homomorphism and a Banach space isomorphism (notice that this notation differs from the one used in  \cite{LotzRosenthal78}, where a lattice isomorphism might not have closed range).
We say that a Banach space $X$ is \textit{linearly embeddable} into another Banach space $Y$ whenever there exists an operator $T\colon X \longrightarrow Y$ which is an isomorphism onto its range, i.e.~$T$ is a linear embedding. If, in addition, $X$ and $Y$ are Banach lattices and $T$ is also a lattice homomorphism then $X$ is said to be \textit{lattice embeddable} into $Y$.
For any Banach lattice $X$ and any compact space $K$, we denote by $C(K,X)$ the Banach lattice consisting of all continuous functions $f\colon K\longrightarrow X$ endowed with the norm $\|f\|=\sup_{t\in K} \|f(t)\|_X$. If we do not require the functions to be continuous then we get the Banach lattice $\ell_\infty(K,X)$. We see $C(K,X)$ as a sublattice of $\ell_\infty(K,X)$ in the canonical way. Furthermore, for any $f\in \ell_\infty(K,X)$, by $\supp(f)$ we denote its support, i.e.~$\supp(f):=\{t\in K: f(t)\neq 0 \}$.

\section{Lattice embeddability of $C[0,1]$}
\label{SectionC[0,1]}

In this section we prove Theorem \ref{TheoA}. First, we need a simple criterion to guarantee that $C[0,1]$ is lattice embeddable in a Banach lattice, which uses the concept of projectivity.

\begin{defn}\cite[Definition 10.1]{dPW15}
\label{DefinitionProjectivity}
A Banach lattice $P$ is said to be \textit{projective} if, whenever $X$ is a Banach lattice, $J$ is a closed ideal in $X$ and $Q\colon X \longrightarrow X/J$ is the quotient map, for every lattice homomorphism $T\colon P \longrightarrow X/J$ and every $\varepsilon >0$ there exists a lattice homomorphism $\widehat T\colon P \longrightarrow X$ such that $T=Q \circ \widehat T$ and $\|\widehat T \| \leq (1+\varepsilon )\|T\|$.
\end{defn}

The fact that $C[0,1]$ is projective \cite[Corollary 11.5]{dPW15}, yields the following result:

\begin{lem}
\label{LemmEmbeddabilityC[0,1]}
Let $X$ be a Banach lattice and $J$ be a closed ideal in 
$X$. If $C[0,1]$ is lattice embeddable into $X/J$ then $C[0,1]$ is lattice embeddable into $X$.  
\end{lem}

\begin{proof}
Let $Q\colon X\longrightarrow X/J$ be the quotient map and let $T\colon C[0,1]\longrightarrow X/J$ be a lattice embedding such that $\alpha \|f\|\leq \|Tf\|\leq\|f\|$ for every $f\in C[0,1]$. Since $C[0,1]$ is projective, for every $\varepsilon>0$ there exists a lattice homomorphism $\widehat T\colon C[0,1] \longrightarrow X$ such that $T=Q \circ \widehat T$ and $\|\widehat T \| \leq (1+\varepsilon )\|T\|$. Note that for $f\in C[0,1]$ we have
$$
(1+\varepsilon)\|f\|\geq\|\widehat Tf\| \geq \|Q\|^{-1}\|Q\widehat Tf\|=\|Tf\|\geq \alpha\|f\|.
$$
Therefore, $\widehat T\colon C[0,1]\longrightarrow X$ defines a lattice embedding (with embedding constant as close as we want to that of $T$).
\end{proof}

To prove Theorem \ref{TheoA} we need to show that if $X$ is a Banach lattice and $C[0,1]$ linearly embeds into $X$ then $C[0,1]$ is lattice embeddable into $X$. By Miljutin's Theorem \cite[Theorem 21.5.10]{semadeni}, the linear embeddability of $C[0,1]$ is equivalent to the linear embeddability of $C(\Delta)$. Now, by \cite[Theorem (A)]{Ghoussoub83}, there exists a positive linear embedding $T\colon C(\Delta) \longrightarrow X$. By the equivalence between (c) and (d) in \cite[Theorem 2]{LotzRosenthal78}, there is no loss of generality in assuming that there is $0<\varepsilon<\frac{1}{4}$ such that  $(1-\varepsilon)\|f\| \leq \|Tf\| \leq \|f\|$ for every $f \in C(\Delta)$.
Thus, Theorem \ref{TheoA} will be a consequence of the following theorem which we prove below.

\begin{thm}
	Let $X$ be a Banach lattice, $0<\varepsilon<\frac{1}{4}$ and $T\colon C(\Delta) \longrightarrow X$ be a positive linear embedding with $(1-\varepsilon)\|f\| \leq \|Tf\| \leq \|f\|$ for every $f \in C(\Delta)$. Then $C[0,1]$ is lattice embeddable into $X$.
\end{thm} 
\begin{proof}
	Without loss of generality we assume that $X$ is the Banach lattice generated by $T(C(\Delta))$. In particular, it is separable, so by \cite[Theorem 1.1]{LLOT19} we can suppose that $X$ is a sublattice of $C(\Delta,L_1)$.
	Let us give a brief overview of the proof. Instead of working directly with $C[0,1]$ we are going to construct a sublattice $Z\subseteq X$, a suitable compact space $K\subseteq \Delta$ and lattice homomorphisms $Q_1\colon C(\Delta,L_1) \longrightarrow C(K,L_1)$ and $Q_2\colon \ell_\infty(K,L_1) \longrightarrow \ell_\infty(K,L_1)$ such that $Q_2(Q_1(Z))$ is a Banach lattice isomorphic to $C(\Delta)$. Since $C[0,1]$ is lattice embeddable into $C(\Delta)$, the conclusion will follow from Lemma \ref{LemmEmbeddabilityC[0,1]}.
	We shall split  the proof into five steps for simplicity. 
	
	\bigskip
	\textbf{Step 1. Writing $C(\Delta)$ as the closure of an increasing sequence of finite dimensional sublattices.}

	We identify $\Delta$ with $2^\omega$ with the product topology. By $2^{<\omega}$ we denote the family of all functions $\sigma \colon \{1,2,\ldots,n\} \longrightarrow \{0,1\}$ for some $n < \omega$. For any such function $\sigma$ we write  $\supp(\sigma)=\{1,2,\ldots,n\}$ and $|\sigma|=n$. Moreover, for any $i \in \{0,1\}$ we denote by $\sigma ^\frown i$ the map defined on $\{1,\ldots,n,n+1\}$ which coincides with $\sigma$ on $\supp(\sigma)$ and takes the value $i$ on $n+1$. For each $\sigma, \sigma' \in 2^{<\omega}$ we write  $\sigma \preceq \sigma '$ if $\sigma'$ extends $\sigma$ (that is, if $\supp(\sigma)\subset \supp(\sigma')$ and $\sigma'|_{\supp(\sigma)}=\sigma$). Thus, $2^{<\omega}$ with this order is a partially ordered set with a least element denoted by $\emptyset$, which is the only function whose domain is the empty set.

	Set $\Delta_\sigma=\{ t\in \Delta: t|_{\supp(\sigma)}=\sigma \}$ and $f_\sigma =\chi_{\Delta_\sigma}$ its characteristic function for every $\sigma \in 2^{<\omega}$. Since each $\Delta_\sigma $ is clopen we have that each $f_\sigma$ is a positive norm-one continuous function in $C(\Delta)$.
	Moreover, $f_\sigma=f_{\sigma^\frown 0}+f_{\sigma^\frown 1}$ and $f_{\sigma^\frown 0},~f_{\sigma^\frown 1}$ are pairwise disjoint for every $\sigma$. 
	It follows that $\{f_\sigma: |\sigma|=n\}$ are pairwise disjoint functions and that $f_\beta \in \spn\{f_\sigma: |\sigma|=n\}$ whenever $|\beta|\leq n$. Thus, the sublattice generated by $\{f_\sigma: |\sigma|\leq n\}$ is just the Banach lattice generated by $\{f_\sigma: |\sigma|=n\}$. Since the Banach lattice generated by finitely many pairwise disjoint vectors is just the linear span of these vectors, we conclude that the Banach lattice generated by $\{f_\sigma: |\sigma| \leq n\}$ is just $Y_n:=\spn\{f_\sigma: |\sigma|=n\}$.
	$Y_n$ is a Banach lattice of dimension $2^n$, lattice isometric to $\ell_\infty^{2^n}$. 
	If we take $Y=\bigcup_n Y_n$, then the lattice version of the Stone-Weierstrass Theorem (see, for instance, \cite[Theorem 2.1.1]{Meyer}) gives that $\overline{Y}=C(\Delta)$.
	
	\bigskip
	\textbf{Step 2. Construction of a lattice quotient $Q_1\colon  C(\Delta, L_1) \longrightarrow C(K,L_1)$ induced by the restriction map to a closed subset $K$ in $\Delta$.}
	
	Since $X\subseteq C(\Delta, L_1)$ and $T$ is a positive linear embedding, each $g_\sigma:=Tf_\sigma$ is a positive continuous function $g_\sigma \colon \Delta \longrightarrow L_1$. By hypothesis, $\|g_\sigma\|\geq 1-\varepsilon$ for each $\sigma \in 2^{<\omega}$.
	Let $K_\sigma:=\{t\in \Delta: \|g_\sigma(t)\| \geq 1-\varepsilon \}$. Since the function $t \mapsto \|g_\sigma(t)\|$ is continuous and $\|g_\sigma\| \geq 1-\varepsilon$, we get that $K_\sigma $ is a nonempty closed subset of $\Delta$ for every $\sigma \in 2^{<\omega}$. Furthermore, notice that $K_{\sigma^\frown 0}, K_{\sigma^\frown 1} \subseteq K_{\sigma}$ for every $\sigma$. Indeed, if $t\in K_{\sigma^\frown 0}$, then $\|g_{\sigma^\frown 0}(t)\|\geq 1-\varepsilon$ and since $g_{\sigma^\frown 0} \leq g_\sigma$ and both are positive, we have that $\|g_\sigma(t)\|\geq 1-\varepsilon$, so $t \in K_\sigma$. Moreover, $K_{\sigma^\frown 0} \cap K_{\sigma^\frown 1} =\emptyset$ since if $t\in K_{\sigma^\frown 0} \cap K_{\sigma^\frown 1}$ then, bearing in mind that the norm of $L_1$ is additive on the positive cone, $1\geq \|g_\sigma(t)\|=\|g_{\sigma^\frown 0}(t)+g_{\sigma^\frown 1}(t)\|=\|g_{\sigma^\frown 0}(t)\|+\|g_{\sigma^\frown 1}(t)\|\geq 2(1-\varepsilon)\geq \frac{3}{2}$, which yields a contradiction.
	Thus, we can take $K:= \bigcap_{k \in \N } (\bigcup_{|\sigma|=k}K_{\sigma})$, which is a nonempty closed subset of $\Delta$. 
	We set $Q_1\colon C(\Delta, L_1) \longrightarrow C(K,L_1)$ the lattice quotient given by the restriction map $Q_1f=f|_K$. Set $K_\sigma':=K_\sigma \cap K$ and $h_\sigma :=Q_1 g_\sigma$ for every $\sigma \in 2^{<\omega}$. We point out the following property, which will be used in the following step:
	\begin{equation}
	\label{Prop1}
	\int (h_\emptyset (s)(t)-h_\sigma (s)(t))dt \leq \varepsilon \medspace \mbox{ for every } \sigma \in 2^{<\omega} \mbox{ and every } s \in K_\sigma'.
	\end{equation}
	Indeed, $\int (h_\emptyset (s)(t)-h_\sigma (s)(t))dt=\|h_\emptyset(s)\|-\|h_\sigma(s)\|$ and (\ref{Prop1}) follows from the fact that $1-\varepsilon \leq  \|h_\sigma(s)\| \leq \|h_\emptyset(s)\| \leq 1$ for every $ s \in K_\sigma'$.

	\bigskip
	\textbf{Step 3. Construction of the lattice homomorphism $Q_2 \colon \ell_\infty(K,L_1) \longrightarrow \ell_\infty(K,L_1)$.}
	
	Notice that $h_\sigma(s)\in L_1$ is an equivalence class of functions for every $\sigma \in 2^{<\omega}$ and every $s \in K$. For simplicity, we identify each equivalence class $h_\sigma(s)$ with a representative, still denoted by $h_\sigma(s)$. We know that if $\sigma \preceq \sigma '$ then $h_\sigma(s) \geq h_{\sigma'}(s)$, which means that $h_\sigma(s)(t) \geq h_{\sigma'}(s)(t)$ for almost every $t \in [0,1]$. 
	There is no loss of generality in assuming that the representatives have been taken in such a way that the inequality $h_\sigma(s)(t) \geq h_{\sigma'}(s)(t)$ holds for every $t \in [0,1]$.
	
	For each $\sigma \in 2^{<\omega}$ and $s \in K_\sigma'$  define
	$$ L_\sigma(s):= \{t \in [0,1]: h_\sigma(s)(t) \geq 2(h_\emptyset (s)(t)-h_\sigma(s)(t))\}.$$
	We prove the following inequality
	\begin{equation}
	\label{Prop2}
	\int_{L_\sigma(s)} h_\emptyset (s)(t)dt \geq 1-4 \varepsilon \medspace \mbox{ for every } s \in K_\sigma', ~\sigma \in 2^{<\omega}.	
	\end{equation}
	
	Indeed, 
	$$ \int_{L_\sigma(s)} h_\emptyset (s)(t)dt = \int_{[0,1]}h_\emptyset(s)(t)dt - \int_{L_\sigma(s)^c} h_\emptyset (s)(t)dt $$
	$$= \|h_\emptyset(s)\| - \int_{L_\sigma(s)^c}( h_\sigma (s)(t) + (h_\emptyset (s)(t)-h_\sigma (s)(t)))dt $$
	$$ \overset{(*)}{\geq} \|h_\emptyset(s)\|- \int_{L_\sigma(s)^c} 3(h_\emptyset (s)(t)-h_\sigma (s)(t))dt \overset{(**)}{\geq} 1-4\varepsilon,$$ 
	where in $(*)$ we have used the definition of $L_\sigma(s)$ and in $(**)$ we have used (\ref{Prop1}) and the fact that $\|h_\emptyset(s)\|\geq 1-\varepsilon$.
	
	Notice that if $\sigma \preceq \sigma '$ and $s \in K_{\sigma'}'$ then $L_{\sigma'}(s) \subseteq L_{\sigma}(s)$, since if $t \in L_{\sigma'}(s)$ then
	$$ h_\sigma(s)(t) \geq h_{\sigma'}(s)(t) \geq 2(h_\emptyset (s)(t)-h_{\sigma'}(s)(t))\geq 2(h_\emptyset (s)(t)-h_{\sigma}(s)(t)).$$
	
	We define for every $s\in K$ the set 
	
	$$ L(s):= \bigcap_{\sigma :~ s\in K_\sigma'} L_\sigma(s).$$ 
	
	Notice that for every $s\in K$ there exists a unique $\sigma_n$ with $|\sigma_n|=n$ such that $s \in K_{\sigma_n}'$. Thus, $L(s)=\bigcap_n L_{\sigma_n}(s)$ and since this is a decreasing sequence of sets, we have that
	\begin{equation}
	\label{Prop3}
	\int_{L(s)} h_\emptyset (s)(t)dt= \int h_\emptyset (s)(t) \chi_{L(s)}(t)dt= \lim_{|\sigma|\rightarrow \infty, s\in K_\sigma'}  \int h_\emptyset (s)(t) \chi_{L_\sigma(s)}(t)dt\geq 1-4\varepsilon	
	\end{equation}
  for every $s \in K$,	by the Monotone Convergence Theorem, where in the last inequality we have used (\ref{Prop2}).

	We consider now the map $Q_2\colon \ell_\infty(K,L_1) \longrightarrow \ell_\infty(K,L_1)$ given by the formula
	$(Q_2f)(s)=f(s)\chi_{L(s)}$ for every $s \in K$. It is easily checked that $Q_2$ is a lattice homomorphism of norm 1 with closed range. The range consists of the functions whose value at each $s$ is supported on $L(s)$.
	Since $C(K,L_1)$ is a sublattice of $\ell_\infty(K,L_1)$, we can define
	$x_\sigma=Q_2 h_\sigma$ for every $\sigma \in 2^{<\omega}$.

	\bigskip

	\textbf{Step 4. Construction of a copy of $C(\Delta)$ in $Q_2(Q_1(X))$.}
	
	We consider the sequence $(y_\sigma)_{\sigma \in 2^{<\omega}}$ in $Q_2(Q_1(X))$ defined recursively as follows:
	\begin{itemize}
		\item $y_\emptyset=x_\emptyset$;
		\item $y_{\sigma^\frown 0} = (2x_{\sigma^\frown 0}-x_{\sigma^\frown 1})^+ \wedge y_\sigma$.
		\item $y_{\sigma^\frown 1} = (2x_{\sigma^\frown 1}-x_{\sigma^\frown 0})^+ \wedge y_\sigma$.
	\end{itemize}
	
	It is immediate that every $y_\sigma$ is positive. We prove the following claims:
	
	\noindent\textbf{Claim 1. $\supp(y_\sigma)\subset K_\sigma'$ and therefore $y_{\sigma^\frown 0}\wedge y_{\sigma^\frown 1}=0$.}

	The claim is trivial for $\sigma=\emptyset$. 
	Suppose the claim holds for $\sigma$ and we show that it is true for $\sigma^\frown 0$ (the case $\sigma^\frown 1$ is analogous).
	By the induction hypothesis, $\supp(y_{\sigma^\frown 0}) \subseteq \supp(y_\sigma) \subseteq K_\sigma'=K_{\sigma^\frown 0}' \bigcup K_{\sigma^\frown 1}'$, so we only need to show that if $s \in K_{\sigma^\frown 1}'$ then $y_{\sigma^\frown 0}(s)(t)=0$ for every $t \in [0,1]$. By the definition of $Q_2$, it is immediate that if $t \notin L(s)$ then  $y_{\sigma^\frown 0}(s)(t)=0$. Nevertheless, since $s\in K_{\sigma^\frown 1}'$, if $t \in L(s)\subseteq L_{\sigma^\frown 1}(s)$ then 
	$$h_{\sigma^\frown 1}(s)(t) \geq 2(h_\emptyset(s)(t)-h_{\sigma^\frown 1}(s)(t) )\geq 2 h_{\sigma^\frown 0}(s)(t).$$
	But then $x_{\sigma^\frown 1}(s)(t) \geq 2 x_{\sigma^\frown 0}(s)(t)$ and therefore $y_{\sigma^\frown 0}(s)(t)=(2x_{\sigma^\frown 0}(s)(t)-x_{\sigma^\frown 1}(s)(t))^+ \wedge y_\sigma=0 \wedge( y_\sigma(s)(t))=0$ as desired.
	
	\noindent \textbf{Claim 2. $y_{\sigma}(s)=x_\emptyset(s)$ for every $\sigma\in 2^{<\omega}$, $s\in K_\sigma'$.}
	The result is trivial for $\sigma=\emptyset$. Suppose the claim holds for a fixed $\sigma$. We show that it holds for $\sigma^\frown 0$ (the case $\sigma^\frown 1$ is analogous).
	It is clear that $y_{\sigma^\frown 0}\leq x_\emptyset$, hence we prove the converse inequality.
	Notice that if $s \in K_{\sigma^\frown 0}'$ and $t\in L_{\sigma^\frown 0}(s)$ then
	$$ h_{\sigma^\frown 0}(s)(t)\geq 2(h_\emptyset(s)(t)-h_{\sigma^\frown 0}(s)(t))=(h_\emptyset(s)(t)-h_{\sigma^\frown 0}(s)(t))+(h_\emptyset(s)(t)-h_{\sigma^\frown 0}(s)(t))\geq $$
	$$ \geq (h_\emptyset(s)(t)-h_{\sigma^\frown 0}(s)(t))+ h_{\sigma^\frown 1}(s)(t), $$
	where in the last inequality we have used that $$ (h_\emptyset-h_{\sigma^\frown 0})=\sum_{\sigma'\neq \sigma^\frown 0, ~|\sigma'|=|\sigma|+1} h_{\sigma'} \geq h_{\sigma^\frown 1}.$$
	
	Thus, we obtain that 
	$$ 2 h_{\sigma^\frown 0}(s)(t) - h_{\sigma^\frown 1}(s)(t) \geq h_\emptyset(s)(t),$$
	for $s \in K_{\sigma^\frown 0}'$ and $t\in L_{\sigma^\frown 0}(s)$. By the definition of $Q_2$ and $y_{\sigma^\frown 0}$ we get that
	$$ y_{\sigma^\frown 0}(s)(t)=( 2 h_{\sigma^\frown 0}(s)(t) - h_{\sigma^\frown 1}(s)(t)) \wedge y_\sigma(s)(t) \geq h_\emptyset(s)(t) \wedge y_\sigma(s)(t) $$
	and the conclusion follows from the induction hypothesis.
	
	\bigskip
	
	The following claim is a direct consequence of Claims 1 and 2, where $\chi_{K_\sigma'}: K \longrightarrow L_1$ is the function which is null on $K\setminus {K_\sigma'}$ and is the constant function one on ${K_\sigma'}$.

	\noindent \textbf{Claim 3. $y_\sigma =\chi_{K_\sigma'} x_\emptyset $ and therefore $y_{\sigma}=y_{\sigma^\frown 0}+y_{\sigma^\frown 1}$ for every $\sigma\in 2^{<\omega}$.}

	\noindent \textbf{Claim 4. $1 \geq \|y_\sigma \| \geq 1-4\varepsilon$ for every $\sigma$.}
	The inequality $1 \geq \|y_\sigma \|$ is immediate, whereas the inequality $\|y_\sigma \| \geq 1-4\varepsilon$ follows from Claim 3 and (\ref{Prop3}).

	Finally, by Claims 1 and 3 we get that the sublattice generated by $\widehat Y_n:=\{y_\sigma: |\sigma|\leq n\}$ is just the Banach lattice generated by $\{y_\sigma: |\sigma|=n\}$. As happened at Step 1,  this Banach lattice is just the linear span of these vectors, i.e. $\widehat Y_n=\spn\{y_\sigma: |\sigma|=n\}$. In fact, if we denote by $\widehat Y$ the Banach lattice generated by $\{y_\sigma: \sigma \in 2^{<\omega}\}$, then it follows from Claim 4 that the unique operator $\widehat T\colon C(\Delta) \longrightarrow \widehat Y$
	satisfying $\widehat T(f_\sigma)=y_\sigma$ is a well-defined lattice isomorphism onto $\widehat Y$.
	\bigskip
	
	\textbf{Step 5. Embedding $C[0,1]$ into $X$.}
	
	Consider the sequence $(z_\sigma)_{\sigma \in 2^{<\omega}}$ in $X$ defined recursively as follows:
	\begin{itemize}
		\item $z_\emptyset=g_\emptyset$;
		\item $z_{\sigma^\frown 0} = (2g_{\sigma^\frown 0}-g_{\sigma^\frown 1})^+ \wedge z_\sigma$.
		\item $z_{\sigma^\frown 1} = (2g_{\sigma^\frown 1}-g_{\sigma^\frown 0})^+ \wedge z_\sigma$.
	\end{itemize}
	
	Since $x_\sigma= Q_2h_\sigma=Q_2(Q_1g_\sigma)$ for every $\sigma \in 2^{<\omega}$, it is immediate that $Q_2(Q_1(z_\sigma))=y_\sigma$  for every $\sigma \in 2^{<\omega}$. Let $Z$ be the sublattice of $X$ generated by $(z_\sigma)_{\sigma \in 2^{<\omega}}$.
	Then $Q_2(Q_1(Z))\subseteq \widehat Y$. We are going to show that $Q_2(Q_1(Z))= \widehat Y$. Let $n\geq 2$ and $a_\sigma $ be a family of scalars with $|\sigma| = n$.
	On one hand, $$\left| \sum_{|\sigma|= n} a_\sigma z_\sigma \right| \leq \sum_{|\sigma|= n} |a_\sigma| z_\sigma  \leq  \sup_{|\sigma|= n} |a_\sigma| \sum_{|\sigma|= n} z_\sigma  \leq  \sup_{|\sigma|= n} |a_\sigma| \sum_{|\sigma|= n-1} (z_{\sigma^\frown 0} + z_{\sigma^\frown 1} ) \leq$$
	$$ \leq   \sup_{|\sigma|= n} |a_\sigma| \sum_{|\sigma|= n-1} ( (2g_{\sigma^\frown 0}-g_{\sigma^\frown 1})^+ +  (2g_{\sigma^\frown 1}-g_{\sigma^\frown 0})^+) \leq   \sup_{|\sigma|= n} |a_\sigma| \sum_{|\sigma|= n-1} ( 2g_{\sigma^\frown 0} +  2g_{\sigma^\frown 1})=$$
	$$= 2 \sup_{|\sigma|= n} |a_\sigma| \sum_{|\sigma|= n} g_{\sigma} =  2 \sup_{|\sigma|= n} |a_\sigma| g_\emptyset .$$
	Thus, $$ \norm{\sum_{|\sigma|= n} a_\sigma z_\sigma} \leq 2 \sup_{|\sigma|= n} |a_\sigma|\| g_\emptyset \| \leq 2 \sup_{|\sigma|= n} |a_\sigma|.$$
	On the other hand,
	by Claims 3 and 4 we have that 
	$$ (1-4\varepsilon)\sup_{|\sigma|= n} |a_\sigma| \leq \norm{\sum_{|\sigma|= n} a_\sigma y_\sigma} \leq \sup_{|\sigma|= n} |a_\sigma|.$$
	
	Hence, it follows from the previous inequalities that $B_{\widehat{Y}} \subseteq \frac{2}{1-4\varepsilon} \overline{Q_2(Q_1(B_{Z}))}$.
	Then, by the classical Banach Open Mapping Theorem (see, for instance, \cite[Lemma 2.23]{fab}), we conclude that $Q_2(Q_1(Z))=\widehat{Y}$ as desired.

	Let $J=\ker((Q_2 \circ Q_1)|_Z)$, which is a closed ideal in $Z$. Then, $Z/J$ is lattice isomorphic to $\widehat Y$, which is in turn lattice isomorphic to $C(\Delta)$ by Step 4. Any surjection $\pi \colon  \Delta \longrightarrow [0,1]$ induces a lattice embedding $S\colon C[0,1] \longrightarrow C(\Delta)$ through the formula $Sf=f\circ \pi$ for every $f \in C[0,1]$. Thus, $C[0,1]$ is lattice embeddable into $C(\Delta)$ and therefore it is also lattice embeddable into $Z/J$. By Lemma \ref{LemmEmbeddabilityC[0,1]}, $C[0,1]$ is lattice embeddable into $Z \subseteq X$, which concludes the proof.
\end{proof}

As a consequence of Theorem \ref{TheoA}, \cite[Theorem 2]{LotzRosenthal78} and \cite[Theorem (A)]{Ghoussoub83}, we get the following corollary:

\begin{cor}
Let $X$ be a Banach lattice. Then the following assertions are equivalent:
\begin{enumerate}
	\item $C[0,1]$ is linearly embeddable into $X$;
	\item $C[0,1]$ is lattice embeddable into $X$;
	\item $L_1$ is lattice embeddable into $X^*$;
	\item $C(\Delta)^*$ is lattice embeddable into $X^*$;
	\item There is an order bounded sequence in $X$ with no weak Cauchy subsequence.
\end{enumerate}
\end{cor}

The following corollary is a simple consequence of Theorem \ref{TheoA}, \cite[Theorem II.2]{GhoussoubJohnson87} and the fact that $\ell_1$ embeds complementably in a Banach lattice if and only if it is lattice embeddable \cite[Theorem 4.69]{AB}.
\begin{cor}\label{cor:dichot}
Let $X$ be a Banach lattice such that $\ell_1$ is linearly embeddable in $X$. Then $\ell_1$ is lattice embeddable in $X$ or $C[0,1]$ is lattice embeddable in $X$.	
\end{cor}


\section{Lattice embeddability of other Banach lattices}
\label{SectionEmbeddability}

We wonder in this section what other Banach lattices, different from $c_0$ and $C[0,1]$, satisfy the property exhibited in Theorem \ref{TheoA}. Before giving a complete characterization in the separable case, we prove Theorem \ref{TheoC}, which will be the key ingredient in the proof of Theorem \ref{TheoB}.

\begin{proof}[Proof of Theorem \ref{TheoC}] Let $X$ be an infinite-dimensional sublattice of $C[0,1]$. Suppose that $X$ is not isomorphic to $c_0$. By \cite[Lemma 2.7.12]{Meyer}
, $X$ does not have order continuous norm. By \cite[Theorem 2.4.2]{Meyer}, there exists a disjoint sequence $(g_n)_{n\in \N}$ and $g$ in $X^+$ with $g$ an upper bound for the sequence $g_n$ and such that $(g_n)_{n\in \N}$ is not norm-convergent to zero. Without loss of generality, we can suppose $\|g_n\|=1$ for every $n\in \N$. Let $I_n$ be a connected component of the open set $\{t\in [0,1]: g_n(t)>0\}$ where $g_n$ attains its maximum. Then, $I_n=(a_n,b_n)$ is an open interval and there is $c_n \in (a_n,b_n)$ such that $g_n(c_n)=1$. Since $(g_n)_{n\in \N}$ is a disjoint sequence, the intervals $I_n$ are disjoint. Thus, the length of $I_n$ converges to zero. By the uniform continuity of $g$, there exists $n\in \N$ such that the length of the interval $g(I_n)$ is smaller than $\frac{1}{2}$. In particular, since $g \geq g_n$ and $g_n(c_n)=1$, we have that $g(t)\geq \frac{1}{2}$ for every $t \in \overline{I_n}$. We claim that $C[0,1]$ is lattice embeddable into the sublattice $Y$ generated by $g$ and $g_n$.

Let $\{(t_\alpha,t_\alpha',\lambda_\alpha)\}_{\alpha \in \Gamma}$ be the family of all elements $t_\alpha,t_\alpha',\lambda_\alpha \in [0,1]$ such that $g(t_\alpha)=\lambda_\alpha g(t_\alpha')$ and $g_n(t_\alpha)=\lambda_\alpha g_n(t_\alpha')$ for every $\alpha \in \Gamma$. By \cite[Theorem 3]{Kakutani}, 
$$Y=\{f \in C[0,1]: f(t_\alpha)=\lambda_\alpha f(t_\alpha') \mbox{ for every }\alpha\in \Gamma \}.$$
Define $T\colon C[0,1] \longrightarrow Y$ by the formula $$Tf(t)=f\left( \frac{g_n(t)}{g(t)} \right) g(t) \mbox{ whenever } t\in [0,1] \mbox{ and } g(t)\neq 0,$$
and $Tf(t)=0$ otherwise. It is immediate that $Tf \in C[0,1]$ for every $f \in C[0,1]$. A routine computation shows that $Tf \in Y$ for every $f \in C[0,1]$. Thus, $T$ is a lattice homomorphism with $\|T\| \leq \|g\|$.
Let $M:=\sup_{t \in I_n} \frac{g_n(t)}{g(t)} $ (notice that $0<M\leq 1$) and define 
$$Z= \{f \in C[0,1]: f \mbox{ is constant on the interval } \left[M, 1\right] \},$$
which is clearly lattice isomorphic to $C[0,1]$.
We claim that $T|_Z$ is an isomorphism. Indeed, if $f\in Z$ then
$$ \|Tf\|=\sup_{t\in [0,1]} |Tf(t)| = \sup_{t\in [0,1]\setminus g^{-1}(0)}\left|f\left( \frac{g_n(t)}{g(t)} \right)\right| g(t) \geq \sup_{t \in I_n } \left|f\left( \frac{g_n(t)}{g(t)} \right)\right| g(t) \geq$$
$$ \geq \frac{1}{2} \sup_{t \in I_n } \left|f\left( \frac{g_n(t)}{g(t)} \right)\right| =  \frac{1}{2} \|f\|, $$
where the last equality follows from the definition of $Z$ and $M$ and the fact that $\left[0,M\right]= \overline{\{\frac{g_n(t)}{g(t)} : t \in I_n\}}$, by definition of $I_n$.
Thus, $T|_Z$ yields a lattice embedding from $Z$ into $Y \subseteq X$, which concludes the proof.
\end{proof}

\bigskip

\begin{proof}[Proof of Theorem \ref{TheoB}] 
Suppose first that $X$ satisfies (1). Then, since every separable Banach space is linearly embeddable into $C[0,1]$, it follows that $X$ is lattice embeddable into $C[0,1]$.

We show now that (2) implies (1). We do know that (1) holds if $X$ is lattice isomorphic to $c_0$ or finite-dimensional. We suppose then that $X$ is an infinite-dimensional sublattice of $C[0,1]$ not lattice isomorphic to $c_0$. By Theorem \ref{TheoC}, $C[0,1]$ is lattice embeddable into $X$. Let $Y$ be an arbitrary Banach lattice containing an isomorphic copy of $X$. Then, it contains an isomorphic copy of $C[0,1]$, so by Theorem \ref{TheoA}, $C[0,1]$ is lattice embeddable into $Y$. Nevertheless, since $X$ is lattice embeddable into $C[0,1]$ we conclude that $X$ is lattice embeddable into $Y$ as desired. 
\end{proof}

After Theorem \ref{TheoB}, one may wonder what are the Banach sublattices of $C[0,1]$. For $C(K)$ spaces we can give a characterization. Recall that a compact space is said to be a \textit{Peano compactum} if it is a continuous image of the interval $[0,1]$.
\begin{prop}
	\label{Peanoproposition}	
	For a compact space $K$, $C(K)$ is lattice embeddable into $C[0,1]$ if and only if $K$ is a disjoint finite union of Peano compacta.
\end{prop}
\begin{proof}
	We only prove that if $C(K)$ is lattice embeddable into $C[0,1]$ then $K$ is a finite union of Peano compacta, which is the nontrivial implication.
	Let $T\colon C(K) \longrightarrow C[0,1]$ be a lattice embedding. By \cite[Theorem 3.2.10]{Meyer}, we have continuous functions $u \colon [0,1] \longrightarrow [0,\infty)$ and $h\colon [0,1]\setminus u^{-1}(0) \longrightarrow K$ so that the embedding  $T \colon C(K) \longrightarrow C[0,1]$ is given by the formula
	$$T(f)(t) = u(t) f(h(t)) \mbox{ whenever } t\in  [0,1]\setminus u^{-1}(0) $$ and $T(f)(t)=0$ otherwise.
	We must have a constant $C>0$ such that  $\|f\| \leq  C \|Tf\|$ for every $f\in C(K)$.
	This implies that  $h(\{t : u(t)\geq \frac{1}{2C}\}) = K$, because otherwise, taking a function $f$ of norm $1$ that vanishes on $h(\{t : u(t) \geq \frac{1}{2C}\})$ we get a contradiction.
	
	There is a finite union $L$ of closed intervals such that $\{t: u(t)\geq \frac{1}{2C}\} \subseteq L \subseteq u^{-1}((0,\infty))$.
	Therefore $K$ is a finite union of Peano compacta. Consider the maximal unions of these finitely many Peano compacta that are path connected. They constitute a partition of $K$ into compact sets. Each of them is a Peano compactum, because it is easy to see that a path connected finite union of Peano compacta is a Peano compactum.
\end{proof}

We finish with a few comments concerning how a nonseparable Banach lattice satisfying condition (1) in Theorem \ref{TheoB} should look like.
Suppose that $X$ is a Banach lattice such that whenever it is linearly embeddable into a Banach lattice then it is lattice embeddable. There are two distinguished Banach lattices where $X$ linearly embeds in. The first one is $C(B_{X^*})$, with $B_{X^*}$ being the unit ball of $X^*$ endowed with the weak*-topology. Thus, $X$ is a sublattice of a $C(K)$-space or, equivalently, $X$ is  an AM-space by the classical Kakutani-Bohnenblust-Krein Theorem (see, for instance, \cite[Theorem 3.6]{AA}).
On the other hand, $X$ linearly embeds into the free Banach lattice generated by $X$, denoted by $FBL[X]$. This notion was introduced in \cite{ART18} and we recall its definition next:

\begin{defn}
Let $E$ be a Banach space. \textit{The free Banach lattice $FBL[E]$ generated by $E$} is a Banach lattice for which there is an isometry $\phi_E \colon E \longrightarrow FBL[E]$ such that for every Banach lattice $Y$ and every operator $T\colon E \longrightarrow Y$ there exists a lattice homomorphism $\widehat{T} \colon FBL[E] \longrightarrow Y$ such that $T=\widehat{T} \circ \phi_E$ and $\|T\|=\|\widehat{T}\|$.
\end{defn}

Thus, if $X$ is lattice embeddable into a Banach lattice whenever it is linearly embeddable then $X$ must be lattice embeddable into $FBL[X]$. One consequence of this fact is that $X$ must have the $\sigma$-bounded chain condition (see Theorem 1.2 and Definition 1.1 in \cite{APR}). In particular, $X$ cannot be $c_0(\Gamma)$ for any uncountable $\Gamma$.

These remarks and the results in this article motivate the following questions, for which we do not know the answer.
\begin{prob}\label{problemuncountablepower}
Is there a nonseparable Banach lattice $X$ such that $X$ is lattice embeddable in a Banach lattice whenever it is linearly embeddable? Does $C([0,1]^\Gamma)$ satisfy this property for some uncountable set $\Gamma$?
\end{prob}

By Theorem \ref{TheoB} and Propostion~\ref{Peanoproposition}, $C([0,1]^\Gamma)$ does have this property when $\Gamma$ is finite or countable, because the finite or countable powers of the interval are Peano compacta.  If we restrict our attention to embeddings inside $C(K)$-spaces, we may ask the following.

\begin{prob}\label{problemCKversion}
	If $C([0,1]^\Gamma)$ linearly embeds into $C(K)$, does $C([0,1]^\Gamma)$ lattice embed into $C(K)$?
\end{prob}

It follows from a result of Plebanek \cite[Corollary 4.6]{Plebanekpositiveembeddings} that there is a lattice embedding from $C([0,1]^\Gamma)$ into $C(K)$ if and only if there is a positive linear embedding, if and only if $K$ maps continuously onto $[0,1]^\Gamma$. Haydon \cite{Haydonl1tau}, Plebanek \cite{Plebanek97} and Fremlin \cite{Fremlin} found various conditions on the cardinality of $\Gamma$ that imply that if $\ell_1(\Gamma)$ linearly embeds into $C(K)$, then $K$ maps continuously onto $[0,1]^\Gamma$, which in turn ensures a positive answer to Problem \ref{problemCKversion}, cf. \cite{Plebaneksurvey}. The question for arbitrary uncountable $\Gamma$ seems however open.

On the other hand,  $C([0,1]^\Gamma)$ does lattice embed into $FBL[C([0,1]^\Gamma)]$, in fact as a complemented sublattice, cf. \cite[Proposition 2.2 and Corollary 3.8]{GKprinciple} and \cite[Theorem 1.4]{amr20_2}.

\section*{Acknowledgements}

We wish to thank Jos\'e Rodr\'iguez for some valuable suggestions.

Research partially supported by Fundaci\'{o}n S\'{e}neca [20797/PI/18]. The first three authors were also supported by project MTM2017-86182-P (Government of Spain, AEI/FEDER, EU). The research of G. Mart\'inez-Cervantes was co-financed by the European Social Fund and the Youth European Initiative under Fundaci\'on S\'eneca  [21319/PDGI/19]. The research of A. Rueda Zoca was also supported by Juan de la Cierva-Formaci\'on fellowship FJC2019-039973, by MICINN (Spain) Grant PGC2018-093794-B-I00 (MCIU, AEI, FEDER, UE), by Junta de Andaluc\'ia Grant A-FQM-484-UGR18 and by Junta de Andaluc\'ia Grant FQM-0185.
P.~Tradacete was also partially supported by Agencia Estatal de Investigaci\'on (AEI) and Fondo Europeo de Desarrollo Regional (FEDER) through grants MTM2016-76808-P (AEI/FEDER, UE) and MTM2016-75196-P (AEI/FEDER, UE), as well as Ministerio de Innovaci\'on y Ciencia, through ``Severo Ochoa Programme for Centres of Excellence in R\&D'' (CEX2019-000904-S).

\end{document}